\documentclass[12pt, reqno]{amsart}
\usepackage{amssymb,a4}
\usepackage{amsmath,amsthm,amscd}
\usepackage{amsthm}
\usepackage{amsfonts,amssymb}
\usepackage{mathrsfs}
\usepackage{amsmath}

\hoffset \voffset \oddsidemargin=60pt \evensidemargin=45pt
\def\a{\alpha}
\def\b{\beta}

\def\Z{\mathbb{Z}}
\def\C{\mathbb{C}}

\def\a{{\alpha}}

\def\c{{\mathbb C}}

\def\z{{\mathbb Z}}

\def\bN{{\mathbb N}}
\def\bZ{{\mathbb Z}}

\def\bC{{\mathbb C}}

\def\b{\beta}

\def\um{\underline{m}}
\def\un{\underline{n}}
\def\ur{\underline{r}}
\def\us{\underline{s}}
\def\ui{\underline{i}}
\def\uj{\underline{j}}

\numberwithin{equation}{section}
\newtheorem{theo}{Theorem}[section]

\newtheorem{coro}[theo]{Corollary}
\newtheorem{lemm}[theo]{Lemma}

\newtheorem{rema}{Remark}

\def\a{\alpha}
\def\b{\beta}
\def\l{\lambda}

\def\1{{\bf 1}}

\title{$U(\frak h)$-free modules over the Lie algebras of differential operators}
\date{}

\author{ Munayim Dilxat}
\address{College of Mathematics and System Sciences, Xinjiang University, Urumqi 830046, Xinjiang, China}
\email{18396830969@163.com}

\author{Shoulan Gao}
\address{Department of Mathematics, Huzhou University, Zhejiang Huzhou, 313000, China}
\email{gaoshoulan@zjhu.edu.cn}

\author{Dong Liu}
\address{College of Mathematics and System Sciences, Xinjiang University, Urumqi 830046, Xinjiang, China; Department of Mathematics, Huzhou University, Zhejiang Huzhou, 313000, China}
\email{liudong@zjhu.edu.cn}

\author{Limeng Xia}
\address{Institute of Applied System Analysis, Jiangsu University, Jiangsu Zhenjiang, 212013,
China}\email{xialimeng@ujs.edu.cn}

\thanks{$^\star$ D. Liu, Corresponding author}

\begin{document}

\maketitle

\noindent{{\bf Abstract.} In this paper, we consider some non-weight modules over the Lie algebra of Weyl type.
First, we determine the modules whose restriction to $U(\frak h)$ are free of rank $1$ over the Lie algebra of differential operators on the circle.  Then we determine the necessary and sufficient conditions for the tensor products of quasi-finite highest weight modules and $U(\frak h)$-free modules to be irreducible, and obtain that any two such tensor products are isomorphic if and only if the corresponding highest weight modules and $U(\frak h)$-free modules are isomorphic. Finally, we extend such results to the Lie algebras of differential operators in the general case.

\noindent{\bf Key words:} }Virasoro algebra; differential operator; polynomial module; Weyl algebra

\parskip .001 truein\baselineskip 6pt \lineskip 6pt

\section{Introduction}

\setcounter{section}{1}\setcounter{equation}{0}

Extended symmetries have played an important role in conformal
field theories. This leads one to study $W$ algebras,
which are higher-spin extensions of the Virasoro algebra, and
their superanalogues, if one takes supersymmetry into account, cf.
\cite{BS, K} and the vast literatures therein. A fundamental
example of $W$ algebras, now known as the $W_{1+\infty}$
algebra, appears as the limit of the
$W_N$ algebras, as $N$ goes to $\infty$ \cite{PRS1}. It is well known that the $W_{1+\infty}$ algebra can
also be interpreted as the central extension of the Lie algebra of
differential operators on the circle \cite{KR, PRS2}. Weight modules over $W_{1+\infty}$ were studied sufficiently and have appeared in various models of two-dimensional quantum field theory and integrable systems (see \cite{FKRW, KR, HL, Su}, etc.). However,  there are few papers on non-weight modules over $W_{1+\infty}$, up to now.

Recently, many authors constructed various non-weight modules over some Lie algebras. In particular, J. Nilsson \cite{N1} constructed a class of ${\frak sl}_n$-modules that are free of rank one when restricted to the enveloping algebra of its Cartan subalgebra. This kind of non-weight modules, which are called $U(\frak h)$-free modules, have been extensively studied. In \cite{N1} and a subsequent paper \cite{N2}, J. Nilsson showed that a finite dimensional simple Lie algebra has nontrivial $U(\frak h)$-free modules if and only if it is of the special linear algebra $\frak{sl}({n+1}, \bC)$ or the symplectic algebra $\frak{sp}({2n}, \bC)$ for some $n\ge 1$.  And the idea was exploited and generalized to consider modules over infinite dimensional Lie algebras, such as the Witt algebra (the Virasoro algebra) \cite{LZ, TZ}, the twisted Heisenberg-Virasoro algebra and the $W(2,2)$ algebra \cite{CG}, and so on.
The aim of this paper is to determine such modules for the Lie algebra $W_{1+\infty}$ motivated by \cite{CG} and \cite{TZ}.
We get a kind of new simple modules over $W_{1+\infty}$ but they are not modules of the corresponding associative algebras (see (3.1)).
Moreover, such results can be extended to the Lie algebras of differential operators on the Laurent polynomial rings of multi-variables.

The paper is organized as follows. In Section 2, we recall  some necessary definitions and preliminary results.  In Section 3, we determine all module structures on $U(\frak h)$, see Theorem 3.2. In Section 4, we give the necessary and sufficient conditions for the tensor product of a quasi-finite highest weight module $L(h, c)$ and a $U(\frak h)$-free $\mathcal D$-module $\Omega(\lambda, \epsilon)$ to be irreducible, see Theorem 4.1. Furthermore, we show that two such tensor product modules are isomorphic if and only if the corresponding quasi-finite highest weight modules and $U(\frak h)$-free modules are isomorphic, see Theorem 4.2. Consequently, we obtain a lot of new irreducible non-weight modules over the Lie algebra of differential operators. In Section 5, we construct such modules for the Lie algebras of differential operators in the general case.

Throughout this paper, denote by $\mathbb Z,$ $\mathbb Z_{+},$ $\mathbb N$, $\mathbb C$ and $\mathbb C^*$ the sets of integers, positive integers,
non-negative integers, complex numbers and non-zero complex numbers, respectively. All algebras and modules are over the complex number field $\mathbb{C}$.

\section{Basics}

In this section, we recall some necessary definitions and preliminary results.

\subsection{The Lie algebra of differential operators}

Let $\bC[t, t^{-1}]$
be the Laurent polynomial ring over $\bC$, ${\mathcal D}_{ass}=\hbox{Diff}\,{\mathbb{C}}[\,t, t^{-1}\,]$
the associative algebra of all differential operators over ${\mathbb{C}}[\,t, t^{-1}\,]$,
which has a basis  $\{t^m D^n\mid m\in {\mathbb Z}, \  n\in\mathbb N\}$ with multiplications:
\begin{equation*}\label{def21}
(t^{m_1}D^{ n_1})(t^{m_2}D^{ n_2})
=\sum_{i=0}^{n_1}{n_1\choose i}\,m_2^i\,t^{m_1+m_2}D^{ n_1+n_2-i}
\end{equation*}
for all $m_1, \; m_2\in\mathbb Z, \; n_1, \; n_2\in\mathbb N$, where $D=t\frac d{dt}$.

Let ${\mathcal D}$ be the Lie algebra of ${\mathcal D}_{ass}$ under Lie bracket given by
\begin{equation*}\label{def22}
\bigl[t^{m_1}D^{ n_1}, t^{m_2}D^{ n_2}\,\bigr]
=\sum_{i=0}^{n_1}{n_1\choose i}\,m_2^i\,t^{m_1+m_2}D^{ n_1+n_2-i}-\sum_{j=0}^{n_2}{n_2\choose j}\,m_1^j\,t^{m_1+m_2}D^{ n_1+n_2-j}
\end{equation*}
for all $m_1, \; m_2\in\mathbb Z, \; n_1, \; n_2\in\mathbb N$.

\begin{lemm}\cite{Zhao}\label{lemm11} The Lie algebra $\mathcal D$ is generated by $\{t, t^{-1}, D^2\}$.
\end{lemm}

Li \cite{Li} proved $H^2({\mathcal D}, \c)=1$
(also see \cite{KR}).   More precisely, we have the following result.

\begin{lemm}\cite{KR, Li}
 Any non-trivial $2$-cocycle on  ${\mathcal D}$ is equivalent to $\phi$:
\begin{equation*}
\phi(t^{m_1}D^{ n_1}, t^{m_2}D^{ n_2})=\begin{cases}&0, \qquad\hbox{\it if}\quad m_1=0,\\
&(-1)^{n_1+1}\,\delta_{m_1+m_2, 0}\,{1\over 2}\,\sum_{i=1}^{m_1}\,(m_1-i)^{n_1}\,i^{n_2}, \qquad\hbox{\it if}\quad m_1>0,\\
&(-1)^{n_1}\,\delta_{m_1+m_2, 0}\,{1\over 2}\,\sum_{i=m_1}^{-1}\,(m_1-i)^{n_1}\,i^{n_2}, \qquad\hbox{\it if}\quad  m_1<0.
\end{cases}
\end{equation*}
\end{lemm}

Let $\widehat{\mathcal D}$ (or $W_{1+\infty}$) denote the universal (one-dimensional) central extension of the Lie algebra ${\mathcal D}$
by the above $2$-cocycle $\phi$ with a central element $C$.  $\widehat{\mathcal D}$ is $\Z$-graded by $\widehat{\mathcal D}_i={\rm Span}_\C\{t^iD^n\mid n\ge0\}\oplus\delta_{i, 0}\C C$ for $i\in\mathbb{Z}$.

A ${\widehat{\mathcal D}}$-module $V$ is called a {highest} (resp.
{lowest) weight module}, if there exists a nonzero $v
\in V_{\lambda}$ such that

1) $\C v$ is a one dimensional ${\widehat{\mathcal D}}_0$-module;

2) ${\widehat{\mathcal D}}_+v=0 $ (resp. ${\widehat{\mathcal D}}_- v=0 $), where ${\widehat{\mathcal D}}_+=\sum_{i>0}{\widehat{\mathcal D}}_i$, ${\widehat{\mathcal D}}_-=\sum_{i<0}{\widehat{\mathcal D}}_i$.

Here we shall note that the highest weight module defined as above is not always quasi-finite.
Quasi-finite highest or lowest weight modules were studied in \cite{KR}.

Next, we define the Verma module, which is a highest weight module with the free action of $U(\widehat{\mathcal D}_-)$  on the highest weight vector (see \cite{Hum}). More precisely, for any $c\in\C$ and $h=(h_1, h_2, \ldots), h_i\in\C, i\in\Z_+$,  let
$\C \1$ be the one-dimensional module over the subalgebra
${\widehat{\mathcal D}}_+\oplus {\widehat{\mathcal D}}_0$ defined by
\begin{equation*}\label{highest}
\widehat{\mathcal D}_+ \1=0,
C \1=c \1, D^k \1=h_k\1.
\end{equation*}
Then we get the induced $\widehat{\mathcal D}$-module, called Verma module:
$$M(h, c)=U(\widehat{\mathcal D})\otimes_{U(\widehat{\mathcal D}_+ +\widehat{\mathcal D}_0)}\C \1.$$

It is well known that the Verma module $M(h, c)$ has a unique maximal
submodule $J(h, c)$ and the corresponding simple quotient module is
denoted by $L(h, c)$. A nonzero weight vector $u'\in M(h, c)$ is
called a singular vector if $\widehat{\mathcal D}_+ u'=0$. It is clear that
$J(h, c)$ is generated by all homogenous singular vectors in $M(h, c)$ not in $\C \1$, and that $M(h, c)=L(h, c)$ if and only if
$M(h, c)$ does not contain any other singular vectors besides those in $\C \1$.

Introduce the generating series
$$
\Delta_h( x ) =
-\sum^\infty_{n = 0} \frac{x^n}{n!} h_n,
$$ which is a formal power series.
A formal power series is called a {\it quasipolynomial} if it is a finite linear combination of functions of the form $p(x)e^{\alpha x}$, where $p(x)$ is a polynomial and $\alpha\in\C$.

The following results were given in \cite{ KR}:
\begin{theo}\cite{KR}\label{simple1}
For any $c\in\C$ and $h\in {\mathcal D}_0^*$, $L(h, c)$ is quasi-finite if and only if  $$
  \Delta_h( x ) = \frac{\phi ( x )}{e^x - 1},
  $$
  where $\phi(x)$ is a quasipolynomial such that $\phi( 0 ) = 0$.

\end{theo}

\subsection{The subalgebras of  the Lie algebra of differential operators}

The Lie algebra ${\mathcal D}$ has a subalgebra, which is called the Witt algebra,  $\mathcal W=\mbox{Der}\,\bC[t, t^{-1}]$ (also denote by Vect($S^1$), the Lie algebra of all vector fields on the circle). The Virasoro algebra Vir is the universal central extension of $\mathcal W$, with a  basis $\{ L_n=t^{n+1}\frac d{dt}, C| n\in \mathbb{Z}\}$ and  relations
 $$
 [L_m,L_n]=(n-m)L_{m+n}+\delta_{n,-m}\frac{m^{3}-m}{12}C,
 $$
$$
[C,L_m]=0, \,  \forall m, n\in\mathbb{Z}.
$$

The Lie algebra ${\mathcal D}$ has also a subalgebra ${\mathcal D}_1$ generated by $\{I_m:=t^m, L_m:=t^{m}D \mid m\in\z\}$ with the following relations:
\begin{eqnarray*}\label{def1}
&&[L_m, L_n]=(n-m)L_{m+n},\nonumber\\
&&[I_m, I_n]=0,\nonumber\\
&&[L_m, I_n]=nI_{m+n}, \forall m,n\in\mathbb Z.
\end{eqnarray*}
Its universal central extension is called the  twisted
Heisenberg-Virasoro algebra (see \cite{ACKP, B, Liu, LPXZ}).

\subsection{$U(\frak h)$-free modules}

In this subsection, we recall some modules which are free of rank $1$ when regarded as a $\C
[L_0]$-module over the Virasoro algebra.

Let $\mathfrak{L}$ be an associative or Lie algebra and $\mathfrak{g}$ be a subspace of $\mathfrak{L}$. A module $V$ over $\mathfrak{L}$ is called $\mathfrak{g}$-torsion if there exists a nonzero $f\in \mathfrak{g}$ such that $fv=0$ for some nonzero $v\in V $. Otherwise, $V$ is called $\mathfrak{g}$-torsion-free.

For $\lambda \in \C^*, \a\in \C,$ denote by
$\Omega(\lambda,\a)=\C[x]$ the polynomial algebra over $\C$. In
\cite{LZ}, the Virasoro algebra $U({\rm  Vir})$-module structure on $\Omega(\lambda,\a)$ is
given by
$$
L_m f(x)=\lambda^m (x-m\a)f(x-m),\quad\,  Cf(x)=0,\quad \forall m\in \Z, f(x)\in\C[x].
$$
By \cite{LZ}, we know that $\Omega(\lambda,\a)$ is simple if and
only if $\lambda,\a\in \C^*.$ If $\a=0,$ then $\Omega(\lambda,0)$
has a simple submodule $x \Omega(\lambda,0)$ with codimension $1.$

\begin{theo} \cite{TZ} \label{Vir}
Let $M$ be a $U({\rm  Vir})$-module such that the restriction of $U({\rm  Vir})$
to $U(\C L_0)$ is free of rank $1$, that is, $M=U(\C L_0)v$ for some
torsion-free $v\in M$. Then $M\cong \Omega(\lambda,\a)$ for some
$\a\in \C, \lambda\in \C^*$.
\end{theo}

The module $\Omega(\l,\a)$ can be naturally made into a $\mathcal D_1$-module (or a module over the twisted Heisenberg-Virasoro algebra) as follows.
 For $\lambda \in \C^*,\a,\b \in \C,$ denote by
$\Omega(\lambda,\a,\b)=\C[x]$ the polynomial algebra over $\C$. We
can define the $\mathcal D_1$-module structure on $\Omega(\lambda,\a,\b)$
as follows
$$
L_m f(x)=\lambda^m (x-m\a)f(x-m), \;
I_m f(x)=\b
\lambda^m  f(x-m),
$$
where $f(x) \in \C[x]$, $m\in \Z$. By \cite{CG}, it
is easy to know that $\Omega(\lambda,\a,\b)$ is simple if and only
if $\a\neq0$ or $\b\neq0$. If $\a=\b=0$, then $\Omega(\lambda,0,0)$
has a simple submodule $x\Omega(\lambda,0,0)$ with codimension
$1.$

\begin{theo} \cite{CG} \label{thw}
Let $M$ be a $\mathcal D_1$-module such that the restriction of $U(\mathcal D_1)$
to $U(\C L_0)$ is free of rank $1$, that is, $M=U(\C L_0)v$ for some
torsion-free $v\in M$. Then $M\cong \Omega(\lambda,\a, \beta)$ for some
$\a, \beta\in \C, \lambda\in \C^*$. Moreover, $M$ is simple if and only
if $\a\neq0$ or $\b\neq0$. If $\a=\b=0$, then $\Omega(\lambda,0,0)$
has a simple submodule $x \Omega(\lambda,0,0)$ with codimension
$1.$
\end{theo}

\section{$U(\frak h)$-free modules over the Lie algebra of differential operators}

In this section, we mainly study $U(\frak h)$-free modules over the Lie algebra of differential operators on the circle based on such researches in \cite{CG}.
We classify all $\mathcal D$-modules such that the restriction of $U(\mathcal D)$
to $U(\C D)$ is free of rank $1$ and get some new irreducible $\mathcal D$-modules.

For $\lambda\in \C^*, \epsilon=0, 1$, denote by
$\Omega(\lambda,\epsilon)=\C[x]$ the polynomial algebra over $\C$. We
can define the actions  of $\mathcal D$ on $\Omega(\lambda, \epsilon)$
as follows
\begin{equation}\label{mod-def1}
t^mD^n f(x)=\beta^{1-n}\lambda^m (x-\epsilon m)^nf(x-m),
\end{equation}
where $\beta=(-1)^{1-\epsilon}$, $f(x) \in \C[x]$ and $m\in \Z, n\in\mathbb N$.

\begin{rema}
In the case of $\epsilon=1$, the module $\Omega(\lambda, 1)$ is also a $\mathcal D_{ass}$-module. From \cite{LZ} we see that if $V$ is an irreducible $\mathcal D_{ass}$-module on which $\C[t, t^{-1}]$ is torsion, then $V\cong \Omega(\lambda, 1)$.
\end{rema}

\begin{lemm} \label{lemma31} $\Omega(\lambda, \epsilon)$ is an irreducible $\mathcal D$-module.
\end{lemm}

\begin{proof}  For $f(x)\in \mathbb{C}[x], m, m_1\in\Z, n, n_1\in\bN$,
\begin{eqnarray*}
 &&[t^{m}D^{n}, t^{m_{1}}D^{n_{1}}]f(x)=(t^{m}D^{n})(t^{m_1}D^{n_1})f(x)-(t^{m_1}D^{n_1})(t^mD^n)f(x)\\
 &=&\lambda^{m+m_1}\beta^{2-n_1-n}(x-\epsilon m)^{n}(x-m-\epsilon m_1)^{n_1}f(x-m-m_1)\\
&&-\lambda^{m+m_1}\beta^{2-n_1-n}(x-\epsilon m_1)^{n_1}(x-m_1-\epsilon m)^{n}f(x-m-m_1).
\end{eqnarray*}
On the other side,
\begin{eqnarray*}
&& [t^{m}D^{n}, t^{m_{1}}D^{n_{1}}]f(x)
 \\
&=&(\sum_{i=0}^{n}{n\choose i}m_1^{i}t^{m+m_1}D^{n+n_{1}-i}
-\sum_{j=0}^{n_1}{n_{1}\choose j}m^{j}t^{m+m_1}D^{n+n_1-j})f(x)
\\
&=&\sum_{i=0}^{n}{n\choose i}\beta^{1-n-n_1+i}\lambda^{m+m_1}m_1^{i}(x-\epsilon m-\epsilon m_1)^{n+n_{1}-i}f(x-m_1-m)
\\
&&-\sum_{j=0}^{n_1}{n_1\choose j}\beta^{1-n-n_1+j}\lambda^{m+m_1}m^{j}(x-\epsilon m-\epsilon m_1)^{n+n_1-j}f(x-m_1-m)
\\
&=&\beta^{1-n-n_1}\lambda^{m+m_1}(x-\epsilon m+(\beta-\epsilon)m_1)^n(x-\epsilon m-\epsilon m_1)^{n_1}f(x-m_1-m)
\\
&&-\beta^{1-n-n_1}\lambda^{m+m_1}(x-\epsilon m_1+(\beta-\epsilon)m)^{n_1}(x-\epsilon m-\epsilon m_1)^nf(x-m_1-m).
 \end{eqnarray*}
So $\Omega(\lambda, \epsilon)$ becomes  a $\mathcal D$-module since
\begin{eqnarray*}
&&(x-\epsilon m+(\beta-\epsilon)m_1)^n(x-\epsilon m-\epsilon m_1)^{n_1}
\\
&&-(x-\epsilon m_1+(\beta-\epsilon)m)^{n_1}(x-\epsilon m-\epsilon m_1)^n
\\
&=&\beta(x-\epsilon m)^{n}(x-m-\epsilon m_1)^{n_1}-\beta(x-\epsilon m_1)^{n_1}(x-m_1-\epsilon m)^{n},
\end{eqnarray*}
where $\beta=1$ if $\epsilon=1$, and $\beta=-1$ if $\epsilon=0$.

It is clear that $\Omega(\lambda, \epsilon)$ is irreducible since $\lambda\ne0$.
\end{proof}

Now we can give the main result of this section.

\begin{theo} \label{main1}Let $M$ be a $\mathcal D$-module such that the restriction of $U(\mathcal D)$ to $U(\C D)$ is free of rank $1$.
Then $M \cong \Omega(\lambda, \epsilon)$ for some $\lambda\in \C^*$ and $\epsilon\in\{0,1\}$.
\end{theo}
\begin{proof}
Let $M=U(\C D)v=\mathbb C[D]v$ be a $U(\mathcal D)$-module such that the restriction of $U(\mathcal D)$ to $U(\C D)$ is free of rank $1$. By Theorem $2$ in \cite{CG},  we have
\begin{equation}\label{G0}
t^mD f(x)=\lambda^m (x-m\a)f(x-m), \quad\,  t^m f(x)=\beta\lambda^m  f(x-m)
\end{equation}
for some $\alpha, \beta\in\C$, where $f(x) \in \C[x]$, $m\in \Z$.

\noindent{\bf Claim 1.}\ \  $\beta\ne0$.

In fact, if $\beta=0$,  by $[D^2, t^m]=2t^mD+t^m$,  we have $t^mDf(x)=0$ for any $m\in\Z$.
In this case, $M$ becomes a trivial $\mathcal D$-module.

\noindent{\bf Claim 2.}\ \  $D^2f(x)=(\frac1\beta x^2+bx+c)f(x)$ for some $b, c\in\C$.

For any $f(x)\in\C[x]$, $D^2f(D)v=f(D)D^2v$. Now we suppose that $D^2v=g(D)v$ for some $g(x)\in\C[x]$, then $D^2 f(x)=g(x)f(x)$ for any $f(x)\in\C[x]$.  By $[D^{2}, t]=2tD+t$, we get
$D^{2}\cdot t\cdot f(x)-t\cdot D^2\cdot f(x)=(2tD+t)\cdot f(x)$.
Then we can obtain
$\beta[g(x)-g(x-1)]=2x+\beta-2\a.$
It is easy to deduce that
$$g(x)=\frac1{\beta}x^2+bx+c$$
for some $b,c\in\C$.

\noindent{\bf Claim 3.}\ \  $b=c=0$, $\alpha\in\{0, 1\}$, and $\beta=(-1)^{1-\alpha}$.

By $[D^{2}, t^{m}]=2mt^{m}D+m^{2}t^{m}$, we have
\begin{eqnarray*}
&&\beta\lambda^m(\frac1{\beta}x^2+bx+c)f(x-m)-\beta[\frac1{\beta}(x-m)^2+b(x-m)+c]f(x-m)
\\
&=&2m\lambda^m(x-\a m)f(x-m)+m^2\beta\lambda^mf(x-m),
\end{eqnarray*}
which forces that $m^{2}+bm=(2\alpha-\beta)m^{2}$ for all $m\in\mathbb{Z}$. So
\begin{equation}\label{G1}
\beta=2\a-1, \quad\, b=0.
\end{equation}
By $[D^{2}, t^{m}D]=m^{2}t^{m}D+2mt^{m}D^{2}$, we can get
\begin{eqnarray*}
& &2mt^{m}D^{2}\cdot f(x)
\\
&=&D^{2}\cdot t^{m}D \cdot f(x)-t^{m}D \cdot D^{2}\cdot f(x)-m^{2}t^{m}D\cdot f(x)
\\
&=&D^{2}\cdot(\lambda^{m}(x-m\alpha)f(x-m))-(t^{m}D)\cdot(g(x)f(x))
-m^{2}\lambda^{m}(x-m\alpha)f(x-m)
\\
&=&\lambda^{m}(x-m\alpha)f(x-m)\big(g(x)-g(x-m)-m^{2}\big)
\\
&=&2m \beta^{-1} \lambda^{m}(x-m\alpha)^{2}f(x-m).
\end{eqnarray*}
Therefore,
\begin{equation}\label{eqd2}
t^{m}D^{2}f(x)=\beta^{-1}\lambda^{m}(x-\a m)^{2}f(x-m), \quad\, m\neq 0.
\end{equation}
Similarly, by $[D^{2}, t^{m}D^2]=2mt^{m}D^{3}+m^{2}t^{m}D^2$ and (\ref{eqd2}), we can obtain
\begin{equation}\label{eqd3}
t^{m}D^{3}f(x)=\beta^{-2}\lambda^{m}(x-\a m)^{3}f(x-m), \quad\, m\neq 0.
\end{equation}
Since $[t^{-1}D^{2}, tD^2]=4D^{3}$, using (\ref{eqd2}), we get
\begin{eqnarray*}
4D^{3}\cdot f(x)&=&t^{-1}D^{2}\cdot tD^{2}\cdot f(x)- tD^{2}\cdot t^{-1}D^{2}\cdot f(x)
 \\
&=&t^{-1}D^{2}\cdot(\beta^{-1}\lambda(x-\alpha)^{2}f(x-1))
-tD^{2}\cdot(\beta^{-1}\lambda^{-1}(x+\alpha)^{2}f(x+1))
\\
&=&\beta^{-2}(x+\alpha)^{2}(x-\alpha+1)^{2}f(x)-\beta^{-2}(x-\alpha)^{2}(x+\alpha-1)^{2}f(x)
\\
&=& 4\beta^{-2}x(x^{2}-\alpha^{2}+\alpha)f(x).
\end{eqnarray*}
Then
\begin{equation}\label{G2}
D^{3}\cdot f(x)=\beta^{-2}x(x^{2}-\alpha^{2}+\alpha)f(x).
\end{equation}
Using (\ref{G0}) and (\ref{G2}), we can get
\begin{eqnarray}
&&D^{3}\cdot tD \cdot f(x)-tD\cdot D^{3} \cdot f(x) \nonumber
\\
&=& \beta^{-2}\lambda(x-\alpha)(3x^{2}-3x+1-\alpha^{2}+\alpha)f(x-1). \label{G3}
\end{eqnarray}
Using (\ref{G0})--(\ref{eqd3}), we have
\begin{eqnarray}\label{k3}
&&(3tD^3+3tD^{2}+tD)\cdot f(x) \nonumber
\\
&=&\lambda (x-\alpha)[3\beta^{-2}(x-\alpha)^{2}+3\beta^{-1}(x-\alpha)+1]f(x-1) \nonumber
\\
&=&\beta^{-2}\lambda (x-\alpha)[3x^{2}-3x+1+\alpha^{2}-\alpha]f(x-1).
\end{eqnarray}
Since $[D^3, tD]=3tD^3+3tD^{2}+tD$, by (\ref{G3}) and (\ref{k3}), we can deduce that
\begin{equation}\label{G4}
\alpha(\alpha-1)=0.
\end{equation}
Combining  (\ref{G1}) and (\ref{G4}), we have $\alpha=0, \beta=-1$ or $\alpha=\beta=1$.
That is,
\begin{equation*}\label{G5}
\alpha\in\{0, 1\}, \quad\, \beta=(-1)^{1-\alpha}.
\end{equation*}
Using (\ref{G0}),  (\ref{eqd2}) and (\ref{G4}), we have
\begin{eqnarray}\label{G6}
& &(tD^{2})\cdot(t^{-1}D)\cdot f(x)-(t^{-1}D)\cdot(tD^{2})\cdot f(x) \nonumber
\\
&=&(tD^{2})(\lambda^{-1}(x+\alpha)f(x+1))-(t^{-1}D)(\beta^{-1}\lambda(x-\alpha)^{2}f(x-1))\nonumber
\\
&=&\beta^{-1}(x-\alpha)^{2}(x+\alpha-1)f(x)-\beta^{-1}(x+\alpha)(x+1-\alpha)^{2}f(x) \nonumber
\\
&=& \beta^{-1}(-3x^{2}+(2\alpha-1)x-\alpha+\alpha^{2})f(x)\nonumber
\\
&=&(-3\beta^{-1}x^{2}+x)f(x),
\end{eqnarray}
and
\begin{equation}\label{G7}
(-3D^2+D)\cdot f(x)=(-3\beta^{-1}x^{2}-3c+x)f(x).
\end{equation}
Since $[tD^2, t^{-1}D]=-3D^2+D$, by (\ref{G6}) and (\ref{G7}),  it is easy to see that
\begin{equation*}\label{G8}
c=0.
\end{equation*}
Furthermore, \eqref{eqd2} holds for all $m\in\Z$, that is,
\begin{equation*}\label{G9}
t^{m}D^{2}f(x)=\beta^{-1}\lambda^{m}(x-\a m)^{2}f(x-m), \quad\, \forall m\in\mathbb{Z}.
\end{equation*}

\noindent{\bf Claim 4.}\ \ For $m\neq 0$, $t^{m}D^{n} \cdot f(x)=\beta^{1-n}\lambda^{m}(x-m\alpha)^{n}f(x-m)$  for all $n\in\mathbb{N}$.

Now we shall prove the result by induction on $n$ in $t^m D^n$ with $m\neq 0$. Suppose that $t^{m}D^{n}f(x)=\beta^{1-n}\lambda^{m}(x-m\alpha)^{n}f(x-m)$ for any $n\leq k$. Then
\begin{eqnarray}\label{G10}
& &D^{2}\cdot (t^{m}D^{k})\cdot f(x)-(t^{m}D^{k})\cdot D^{2}\cdot f(x) \nonumber
\\
&=&D^{2}\cdot(\beta^{1-k}\lambda^{m}(x-m\alpha)^{k}f(x-m))-(t^{m}D^{k})\cdot(\beta^{-1} x^{2}f(x))\nonumber
\\
&=&\beta^{-k}\lambda^{m}x^{2}(x-m\alpha)^{k}f(x-m)-\beta^{-k}\lambda^{m}(x-m\alpha)^{k}(x-m)^{2}f(x-m) \nonumber
\\
&=&\beta^{-k}\lambda^{m}[x^{2}-(x-m)^{2}](x-m\alpha)^{k}f(x-m) \nonumber
\\
&=&\beta^{-k}\lambda^{m}m(2x-m)(x-m\alpha)^{k}f(x-m).
\end{eqnarray}
Since $[D^{2},t^{m}D^{k}]=m^{2}t^{m}D^{k}+2mt^{m}D^{k+1}$,
by (\ref{G10}), we get
\begin{eqnarray*}
& &2mt^{m}D^{k+1}\cdot f(x)
\\
&=&\beta^{-k}\lambda^{m}m(2x-m)(x-m\alpha)^{k}f(x-m)-m^{2}(t^{m}D^{k})\cdot f(x)
\\
&=&\beta^{-k}\lambda^{m}m(2x-m)(x-m\alpha)^{k}f(x-m)-m^{2}\beta^{1-k}\lambda^{m}(x-m\alpha)^{k}f(x-m)
\\
&=&\beta^{-k}\lambda^{m}m(2x-2m\alpha)(x-m\alpha)^{k}f(x-m)
\\
&=&2m\beta^{1-(k+1)}\lambda^{m}(x-m\alpha)^{k+1}f(x-m).
\end{eqnarray*}
Therefore, $t^{m}D^{k+1}\cdot f(x)=\beta^{1-(k+1)}\lambda^{m}(x-m\alpha)^{k+1}f(x-m)$ for $m\neq 0$.

\noindent{\bf Claim 5.}\ \ $D^{n}\cdot f(x)=\beta^{1-n}x^{n}f(x)$  for all $n\in\mathbb{N}$.

We shall prove the result by induction on $n$ in $ D^n$. Suppose that $D^{n}f(x)=\beta^{1-n}x^{n}f(x)$ for any $n\leq k$.  By
$
[t^{-1}D^{k},tD^{2}]=(k+2)D^{k+1}+\frac12(k+1)(k-2)D^{k}+\sum_{i=3}^{k}{k\choose i}D^{k+2-i},
$
we get
\begin{eqnarray}\label{G11}
&&(k+2)D^{k+1}\cdot f(x) \nonumber
\\
&=&(t^{-1}D^{k})(tD^{2})\cdot f(x)-(tD^{2})(t^{-1}D^{k})\cdot f(x) \nonumber
\\
&&-\frac12(k+1)(k-2)\beta^{1-k}x^{k}f(x)-\sum_{i=3}^{k}{k\choose i}\beta^{i-k-1}x^{k+2-i} f(x)\nonumber
\\
&=&(t^{-1}D^{k})(\beta^{-1}\lambda(x-\alpha)^{2} f(x-1))-(tD^{2})(\beta^{1-k}\lambda^{-1}(x+\alpha)^{k}f(x+1)) \nonumber
\\
&&-\frac12(k+1)(k-2)\beta^{1-k}x^{k}f(x)-\sum_{i=3}^{k}{k\choose i}\beta^{i-k-1}x^{k+2-i} f(x) \nonumber
\\
&=&\beta^{-k}[(x+\alpha)^{k}(x+1-\alpha)^{2}-(x-\alpha)^{2}(x-1+\alpha)^{k}]f(x) \nonumber
\\
& &-\frac12(k+1)(k-2)\beta^{1-k}x^{k}f(x)-\sum_{i=3}^{k}{k\choose i}\beta^{i-k-1}x^{k+2-i} f(x).
\end{eqnarray}
If $\alpha=0$ in (\ref{G11}),  we have
\begin{eqnarray*}
&&(k+2)D^{k+1}\cdot f(x)
\\
&=&\beta^{-k}\left(x^{k}(x+1)^{2}-x^{2}(x-1)^{k}-\frac12(k+1)(k-2)\beta x^{k}
-\sum_{i=3}^{k}{k\choose i}\beta^{i-1}x^{k+2-i} \right) f(x)
\\
&=&\beta^{-k}[x^{k+2}+2x^{k+1}+x^{k}-\sum_{i=0}^{k}{k\choose i}\beta^{i}x^{k+2-i}+\frac12(k+1)(k-2)x^{k}
\\
& &
-\sum_{i=3}^{k}{k\choose i}\beta^{i-1}x^{k+2-i} ] f(x)
\\
&=&\beta^{-k}\left((k+2)x^{k+1}-\sum_{i=3}^{k}{k\choose i}\beta^{i}x^{k+2-i}
-\sum_{i=3}^{k}{k\choose i}\beta^{i-1}x^{k+2-i} \right) f(x)
\\
&=&\beta^{1-(k+1)}(k+2)x^{k+1}f(x).
\end{eqnarray*}
If $\alpha=1$ in (\ref{G11}), we have
\begin{eqnarray*}
&&(k+2)D^{k+1}\cdot f(x)
\\
&=&\beta^{-k}\left(x^{2}(x+1)^{k}-x^{k}(x-1)^{2}-\frac12(k+1)(k-2)x^{k}
-\sum_{i=3}^{k}{k\choose i}x^{k+2-i}\right) f(x)
\\
&=&\beta^{-k}\left(\sum_{i=3}^{k}{k\choose i}x^{k+2-i} +(k+2)x^{k+1}-\sum_{i=3}^{k}{k\choose i}x^{k+2-i}\right) f(x)\\
&=&\beta^{1-(k+1)}(k+2)x^{k+1}f(x).
\end{eqnarray*}
Therefore,
$
t^{m}D^{k+1}\cdot f(x)=\beta^{1-(k+1)}\lambda^{m}(x-m\alpha)^{k+1}f(x-m).
$
Hence the claim holds.

By Claim $4$ and Claim $5$, we obtain
$$t^{m}D^{n} \cdot f(x)=\beta^{1-n}\lambda^{m}(x-m\alpha)^{n}f(x-m), \quad\, \forall m\in\mathbb{Z},n\in\mathbb{N}.$$
The theorem holds.
\end{proof}

\begin{coro}\label{c1}
Let $M$ be a $\widehat{\mathcal D}$-module such that the restriction of $U(\widehat{\mathcal D})$ to $U(\C D)$ is free of rank $1$.
Then $M \cong \Omega(\lambda, \epsilon)$ for some $\lambda\in \C^*$ and $\epsilon=0,1$, where $C$ acts trivially on $M$.
\end{coro}
\begin{proof}
It follows from Theorem \ref{main1} and Theorem \ref{Vir}.
\end{proof}

\section{Tensor Products }

 In this section, we will obtain a class of irreducible non-weight modules over $\mathcal{D}$ by taking the tensor products of $\Omega(\lambda, \epsilon)$ with  quasifinite highest weight modules and determine the necessary and sufficient conditions for two irreducible tensor products to be isomorphic.

\begin{theo}\label{51}  Let $\lambda \in \mathbb{C}^{\ast}$ and $\epsilon \in \{0,1\}$. Let $V$ be an irreducible quasifinite highest weight module over $\mathcal{D}$. Then $\Omega(\lambda, \epsilon)\otimes V$  is an irreducible $\mathcal D$-module.
\end{theo}
\begin{proof} Let $W=\Omega(\lambda, \epsilon)\otimes V$. It is clear that, for any $v\in V$, there is a positive integer $K(v)$ such that $t^{m}D^{n}(v)=0$ for all $m\geq K(v)$.

Suppose $M$ is a nonzero submodule of $W$. It suffices to show that $W=M$. Take a nonzero $w=\sum_{j=0}^{s}x^{j}\otimes v_{j} \in W$ such that $v_{j}\in V$, $v_s\neq 0$ and $s $ is minimal.

\noindent{ {\bf Claim 1.} $s=0$.}

Let $K={\rm max}\,\{K(v_{j}):j=0,1,...,s\}$.  Using $t^{m}D(v_{j})=0$ for all $m\geq K$ and $j=0,1,...,s$, we can deduce that
$$
\lambda^{-m}t^{m}D(w)=\sum_{j=0}^{s}(x-\epsilon m)(x-m)^{j}\otimes v_{j} \in M,
 \, \forall m\geq K.
 $$
Write the right hand side as
\begin{equation*}
\sum_{j=0}^{s+1}m^{j}w_j\in M,  \quad\,    \forall m \geq K,
\end{equation*}
where $w_{j}\in W$ are independent of $m$. In particular, $w_{s+1}=\epsilon (-1)^{s-1}\otimes v_{s}$.  Taking $m=K, K+1, ..., K+s$, we see that the coefficient matrix of $w_{i}$ is a  Vandermonde matrix. So each $w_i\in M$. In particular, $w_{s+1}=\epsilon (-1)^{s-1}\otimes v_{s}\in M$. Consequently,  $s=0$.

\noindent{{\bf Claim 2.} $W=M$.}

By Claim 1, we know that $1\otimes v\in M$ for some nonzero $v\in V$. By induction on $i$ and using
\begin{eqnarray*}
 t^{m}D(x^{i}\otimes v)&=&(\lambda^{m}(x-\epsilon m)(x-m)^{i})\otimes v\\
&=&(\lambda^{m}(x-m)^{i+1})\otimes v+(\lambda^{m}m(1-\epsilon)(x-m)^{i})\otimes v,
\end{eqnarray*}
where $m \geq K(v)$, $i\in \mathbb{N}$, we deduce that $x^{i}\otimes v \in M$ for all $i\in\mathbb{N}$, i.e., $\Omega(\lambda,\epsilon)\otimes v\subset M$. Let $X$ be a maximal subspace of $V$ such that $\Omega(\lambda,\epsilon)\otimes X\subset M$. We know that $X\neq 0$. Clearly, $X$ is a nonzero submodule of $V$. Since $V$ is irreducible,  we obtain that $X=V$. Therefore, $M=W$ and $M$ is irreducible.
\end{proof}

Next, we will determine the necessary and sufficient conditions for two irreducible tensor products defined in Theorem \ref{51} to be isomorphic.

\begin{theo}\label{52}  Let $\lambda,\lambda' \in \mathbb{C}^{\ast}$ and $\epsilon,\epsilon' \in \{0,1\}$, and let $V$ and $V'$  be two irreducible  quasifinite highest weight modules over $\mathcal{D}$. Then $W=\Omega(\lambda, \epsilon)\otimes V$ and  $W'=\Omega(\lambda', \epsilon')\otimes V'$   are isomorphic as  $\mathcal D$-modules if and only if $(\lambda,\epsilon)=(\lambda' , \epsilon')$ and $V\cong V'$ as $\mathcal D$-modules.
\end{theo}
\begin{proof}  The sufficiency of the theorem is obvious. We only need to prove the necessity.
Let $\varphi$ be an isomorphism from $W$ to $W'$ as  $\mathcal D$-modules.
Take a nonzero element $1\otimes v \in W$. Suppose
$$
\varphi(1\otimes v)=\sum_{j=0}^{k}x^{j}\otimes w_{j},
$$
where $w_{j}\in V'$ with $w_{k}\neq 0$. There is a positive integer $K=K(v)$ such that $t^{m}D \cdot v=t^{m}D\cdot w_{j}=0$ for all $m\geq K$ and $0\leq j \leq k$. For any $m,m'\geq K$, we have
$$
(\lambda^{-m}t^{m}D-\lambda^{-m'}t^{m'}D)(1\otimes v)=\epsilon(m'-m)(1\otimes v),
$$
which suggests
\begin{eqnarray}\label{G12}
& &\epsilon (m'-m)(\sum_{j=0}^{k}x^{j}\otimes w_{j})= (\lambda^{-m}t^{m}D-\lambda^{-m'}t^{m'}D)(\sum_{j=0}^{k}x^{j}\otimes w_{j})\nonumber
\\
&=&\sum_{j=0}^{k}\left[(\frac{\lambda'}{\lambda})^{m}(x-\epsilon' m)(x-m)^{j}-
(\frac{\lambda'}{\lambda})^{m'}(x-\epsilon' m')(x-m')^{j}\right]\otimes w_{j}.
\end{eqnarray}
Then we  can get that
$
\left[(\frac{\lambda'}{\lambda})^{m}-(\frac{\lambda'}{\lambda})^{m'}\right](x^{k+1}\otimes w_{k})=0
$ for all $m,m' \geq K.$
So
\begin{equation*}\label{G13}
\lambda=\lambda'
\end{equation*}
and (\ref{G12}) becomes
\begin{equation}\label{G13+1}
\epsilon( m'-m)(\sum_{j=0}^{k}x^{j}\otimes w_{j})
=\sum_{j=0}^{k}((x-\epsilon' m)(x-m)^{j}-(x-\epsilon' m')(x-m')^{j})\otimes w_{j},
\end{equation}
where $m,m'\geq K$. If $k> 0$, for $\epsilon'=1$, by the  coefficient of $m^{k+1}$ in (\ref{G13+1}), we get
$(-1)^{k+1}\epsilon' (1\otimes w_{k})=0$, a contradiction. If $k> 0$, for $\epsilon' =0$,  by the leading coefficient of $m$ in (\ref{G13+1}), it is easy to see that $k=1$ and $\epsilon( m'-m)(1\otimes w_{0}+x\otimes w_{1})
=(m'-m)x\otimes w_{1}$ for all $m,m'\geq K$, which is impossible.  Therefore,
$k=0$ and
$\epsilon( m'-m)(1\otimes w_{k})=\epsilon'( m'-m)(1\otimes w_{k})$ for all $m,m'\geq K$,
which forces
\begin{equation*}\label{G14}
\epsilon=\epsilon'.
\end{equation*}
Thus there is a bijection  $\tau: V \longrightarrow V'$ such that
\begin{equation*}\label{54}
\varphi(1\otimes v)=1\otimes \tau(v),  \quad\,    \forall  v \in V.
\end{equation*}

Since
$
\varphi(t^{m}D^{n}\cdot(1\otimes v))=t^{m}D^{n}\cdot(\varphi(1\otimes v))=t^{m}D^{n}\cdot(1\otimes \tau(v)),
$
we have
\begin{eqnarray*}\label{G15}
& &\beta^{1-n}\lambda^{m}\varphi((x-\varphi m)^{n}\otimes v)+ \varphi(1\otimes(t^{m}D^{n}\cdot v)) \nonumber
\\
&=& \beta^{1-n}\lambda^{m}(x-\varphi m)^{n}\otimes \tau(v)+1\otimes(t^{m}D^{n}\cdot \tau(v)), \quad\,m\in \mathbb{Z}.
\end{eqnarray*}
Hence
$
\beta^{1-n}\lambda^{m}\varphi((x-\epsilon m)^{n}\otimes v)=\beta^{1-n}\lambda^{m}(x-\epsilon m)^{n}\otimes\tau(v)
$ for $m\geq K.$ So
$$
\varphi((x-\epsilon m)^{n}\otimes v)=(x-\epsilon m)^{n}\otimes\tau(v), \quad\,m\geq K.
$$
Then it is easy to deduce that
$$\varphi(t^{m}D^{n}(1)\otimes v)=t^{m}D^{n}(1)\otimes\tau (v), \quad\,m\in \mathbb{Z}.$$
By
$
\varphi(t^{m}D^{n}\cdot(1\otimes v))=t^{m}D^{n}\cdot(\varphi(1\otimes v))
$,
we can obtain
$$\varphi(1\otimes t^{m}D^{n} (v))=1\otimes t^{m}D^{n}(\tau( v)).$$
So
$$
  \tau (t^{m}D^{n}( v))=t^{m}D^{n}(\tau( v)), \,  \, \forall m \in \mathbb{Z}, \, v \in V.
$$
 Clearly, for $c \in \mathbb{C}$, $\varphi(c(1\otimes v))=c(\varphi(1\otimes v))$ implies that $\tau(c v)=c\tau(v)$. Thus $\tau:V\longrightarrow V'$  is a $\mathcal{D}$-module isomorphism and $V\cong V'$. The proof is finished.
\end{proof}

\section{Differential operator algebra on $\bC[t_1^{\pm1}, \ldots, t_\nu^{\pm1}]$}

For any positive integer $\nu$, let $\mathcal A:=\bC[t_1^{\pm1}, \ldots, t_\nu^{\pm1}]$ be the Laurent polynomial algebras. Then we have the Lie algebras $\mathcal W_\nu ={\rm Der} (\mathcal A)$ and the Lie algebra ${\mathcal D}_\nu={\rm Diff}(\mathcal A)$. These Lie algebras are known as the Witt algebra and the differential operator algebra of rank $\nu$. We know that $\mathcal W_1$ is the centerless Virasoro algebra and $\mathcal D_1=\mathcal D$ is the Lie algebra of the differential operators.

For $\um=(m_1, \ldots, m_\nu)\in \bZ^\nu, \un=(n_1, \ldots, n_\nu)\in \bN^\nu$, set $t^{\um} = t_1^{m_1}\ldots t_\nu^{m_\nu}, D^{\un}=D_1^{n_1}\ldots D_\nu^{n_\nu}$, where $D_i=t_i\frac{d}{dt_i}$ for all $i=1,2,\ldots, \nu$. We can write the Lie
brackets in $\mathcal W_\nu$ and $\mathcal D_\nu$ as follows:

\begin{eqnarray}
[t^{\ur}D_i,  t^{\us}D_j]&=&s_it^{\ur+\us}D_j-r_jt^{\ur+\us}D_i,\nonumber\label{gw1}
\\
\bigl[\,t^{\ur}D^{\um}, t^{\us}D^{\un}\,\bigr]
&=&\sum_{i_1=0}^{m_1}\ldots\sum_{i_\nu=0}^{m_\nu}{m_1\choose i_1}\ldots{m_\nu\choose i_\nu}\,\us^{\ui}\,t^{\ur+\us}D^{m+n-i}\nonumber
\\
&-&\sum_{j_1=0}^{n_1}\ldots\sum_{j_\nu=0}^{n_\nu}{n_1\choose i_1}{n_2\choose i_2}\ldots{n_\nu\choose i_\nu}\,\ur^{\uj}\,t^{\ur+\us}D^{ m+n-j}\label{gd1}
\end{eqnarray}
for all $\ur, \us\in\mathbb Z^\nu, \um, \un\in\mathbb N^\nu, \ui=(i_1, \ldots, i_\nu), \uj=(j_1, \ldots, j_\nu)$.

It is known that $\frak h_\nu={\rm span}_{\mathbb{C}}\{D_1, D_2, \ldots, D_\nu\}$ is the Cartan subalgebra of $\mathcal W_\nu$ and $\mathcal D_\nu$.

\begin{lemm} \label{lemm41}
The Lie algebra $\mathcal D_\nu$ is generated by $\{t_i, t_i^{-1}, D_iD_j\mid i, j=1,2, \ldots, \nu\}$.

\end{lemm}
\begin{proof}
It follows by Lemma \ref{lemm11} and \eqref{gd1}.
\end{proof}

In \cite{TZ}, a class of $\mathcal W_\nu$-modules  were defined as follows,
where $1\le\nu<\infty$. For  any $\a\in \C$ and
$\Lambda=(\lambda_1,\lambda_2,\ldots,\lambda_\nu)\in (\C^*)^n$,
denote by $\Omega(\Lambda, \alpha)=\C[x_1,x_2,\ldots,$
 $x_\nu]$ the polynomial  algebra over $\C$ in commuting indeterminates $x_1,x_2,\ldots,
 x_\nu.$ The action of $\mathcal W_\nu$ on $\Omega(\Lambda, \a)$ is defined by
\begin{equation}t^{\um}D_i \cdot
f(x_1,\ldots,x_\nu)=\Lambda^m(x_i-\a m_i)f(x_1-m_1,\ldots,x_\nu-m_\nu),\label{gwdef}\end{equation}
where $\um=(m_1,m_2,\ldots,m_\nu)\in \Z^\nu, f(x_1,\ldots,x_\nu)\in \mathcal A,
\Lambda^{\um}=\lambda_1^{m_1}\lambda_2^{m_2}\ldots\lambda_\nu^{m_\nu},$
$i=1,2,\ldots,\nu.$
 By \cite{TZ}, we know that $\Omega(\Lambda, \a)$ is irreducible if and only if $\a\ne 0.$ The following result was given in \cite{TZ}.

\begin{theo}\cite{TZ}\label{gwp}  Let $M$ be a $\mathcal W_\nu$-module which is a free $U(\frak h_\nu)$-module of  rank $1$. Then
$M\cong \Omega(\Lambda, \a)$ for some $\Lambda\in (\C^*)^\nu$ and some $\a\in\C.$
\end{theo}

For $\epsilon=0, 1$,  $\Lambda=(\lambda_1,\lambda_2,\ldots,\lambda_\nu)\in (\C^*)^\nu$, we define a $\mathcal D_\nu$-module  $\Omega(\Lambda, \epsilon)$ as follows:
Denote by $\Omega(\Lambda, \epsilon)=\C[x_1,x_2,\ldots,$
 $x_\nu]$ the polynomial  algebra over $\C$ in commuting indeterminates $x_1,x_2,\ldots,
 x_\nu$. The action of $\mathcal D_\nu$ on $\Omega(\Lambda, \epsilon)$ is defined by
\begin{equation*}
t^{\underline m}D^{\un} \cdot
f(x_1,\ldots,x_\nu)
=\beta^{1-\sum_{i=1}^\nu n_i}\Lambda^m\Pi_{j=1}^\nu(x_j-\epsilon m_j)^{n_j}f(x_1-m_1,\ldots,x_\nu-m_\nu)
\end{equation*}
for any $ f(x_1,\ldots,x_\nu)\in \mathcal A$, where $\beta=(-1)^{1-\epsilon}$, $\um=(m_1,m_2,\ldots,m_\nu)\in \Z^\nu, \un=(n_1,n_2,\ldots, n_\nu)\in \bN^\nu,$
$i=1,2,\ldots, \nu$.
\begin{lemm}\label{lemma42} $\Omega(\Lambda, \epsilon)$ is an irreducible $\mathcal D_\nu$-module.
\end{lemm}
\begin{proof} It follows by Lemma \ref{lemma31} and some direct calculations.
\end{proof}
\begin{theo} \label{main3} Let $M$ be a $\mathcal D_\nu$-module which is a free $U(\frak h_\nu)$-module of rank $1$. Then
$M\cong \Omega(\Lambda, \epsilon)$ for some $\Lambda\in (\C^*)^\nu, \epsilon=0, 1$.
\end{theo}

\begin{proof} By Theorem \ref{main1},  we get
\begin{equation}\label{G16}
t_i^{m_i}D_i^{n_i} \cdot
f(x_1,\ldots, x_\nu)
=\beta^{1-n_i}\lambda_i^{m_i}(x_i-\epsilon m_i)^{n_i}f(x_1,\ldots, x_i-m_i, \ldots,  x_\nu)
\end{equation}
for any $ f(x_1,\ldots,x_\nu)\in \mathcal A$, where $\epsilon\in\{0,1\}$, $\beta=(-1)^{1-\epsilon}$, $m_i\in\bZ$, and $\lambda_i\in\mathbb{C}^{*}$, $i=1, \ldots,\nu$.

Moreover, in this case, \eqref{gwdef} becomes
\begin{equation}\label{G17}
t^{\um}D_i\cdot
f(x_1,\ldots, x_\nu)=\Lambda^{\um}(x_i-\epsilon m_i)f(x_1-m_1,\ldots, x_i-m_i, \ldots,  x_\nu-m_\nu)
\end{equation}
for any $f(x_1,\ldots,x_\nu)\in \mathcal A$,
where $\um=(m_1, m_2,\ldots, m_\nu)\in \Z^\nu$, $\Lambda^{\um}=\lambda_1^{m_1}\ldots\lambda_\nu^{m_\nu}$,
$i=1,2,\ldots, \nu$.

By $[t^{\um}D_i, D_i^2]=-2m_it^{\um}D_i^2-m_i^2t^{\um}D_i$, using (\ref{G16}) and (\ref{G17}), we get
\begin{equation}\label{G18}
t^{\um}D_i^2 \cdot
f(x_1,\ldots, x_\nu)
=\beta^{-1}\Lambda^{\um}(x_i-\epsilon m_i)^2f(x_1-m_1,\ldots, x_i-m_i, \ldots,  x_\nu-m_\nu).
\end{equation}
For any $1\le i\ne j\le \nu$, by $[t_i^{-1}D_i^2, t_iD_j]=2D_iD_j+D_j$, using (\ref{G17}) and (\ref{G18}), we can obtain
\begin{equation*}
D_iD_j\cdot
f(x_1,\ldots, x_\nu)=\beta^{-1}x_ix_jf(x_1,\ldots,  x_\nu), \ \text {for all}\ f(x_1,\ldots,x_\nu)\in \mathcal A.
\end{equation*}
By Lemma \ref{lemm41}, the theorem holds.
\end{proof}

In conclusion, a kind of new simple modules over the Lie algebra of differential operators, which are non-weight modules,  are constructed by classical methods in this paper (Theorem \ref{main1} and Theorem \ref{main3}).
Certainly, the study of non-weight module over the Lie algebra of differential operators is still in its infancy.

\vskip30pt \noindent{\bf Acknowledgments}

This work is partially supported by the NNSF (Nos. 12071405, 11971315, 11871249,12171155), and is partially supported by Xinjiang Uygur Autonomous Region graduate scientific research innovation project (No. XJ2021G021).

\vskip10pt \noindent{\bf  Data availability}  Data sharing not applicable to this paper.

\vskip10pt \noindent{\bf Conflict of interest} The authors declared that they have no conflicts of interest to this work.

\end{document}